\documentclass[twoside,12pt]{article}
\usepackage{amsmath,amssymb,amsbsy,amsfonts,amsthm,latexsym,amsopn,amstext,amsxtra,euscript,amscd}

\usepackage{tikz}
\usetikzlibrary{shapes,fit,arrows,positioning}
\tikzstyle{vertex} = [circle,minimum size=17pt, draw, thick, inner sep=0pt, text centered]
\tikzstyle{edge} = [draw, thick,->]

\usepackage[pdftex]{hyperref}
\usepackage{graphics,graphicx}

\def\ord{{\rm ord}}

\newtheorem{theorem}{Theorem}
\newtheorem{lemma}{Lemma}
\newtheorem{obs}{Observation}
\newtheorem{corollary}{Corollary}

\newcommand*{\TIT}{On sunlet graphs connected to a specific map on $\{1,2,\dots,p-1\}$}
\title{\bf \TIT}

\newcommand*{\ADRo}{Laboratory of Mathematics, Cryptography and Mechanics, University of Hassan II Mohammedia-Casablanca, Morocco. \textit{khadir@hotmail.com}}

\newcommand*{\ADRn}{University of Sopron, Institute of Mathematics, Hungary. \textit{nemeth.laszlo@uni-sopron.hu}}

\newcommand*{\ADRs}{J.~Selye University, Department of Mathematics and Informatics, Slovakia; and University of Sopron,  Institute of Mathematics, Hungary.  \textit{szalay.laszlo@uni-sopron.hu}}

\author{\sc Omar Khadir\footnote{\ADRo}, L\'aszl\'o N\'emeth\footnote{\ADRn}, L\'aszl\'o Szalay\footnote{\ADRs}}
\date{}

\begin{document}

\maketitle \thispagestyle{empty}

\begin{abstract}
In this article, we study the structure of the graph implied by a given map on the set $S_p=\{1,2,\dots,p-1\}$, where $p$ is an odd prime. The consecutive applications of the map generate an integer sequence, or in graph theoretical context a walk, that is linked to the discrete logarithm problem.   \\[1mm]
The final publication is available at Annales Mathematicae et Informaticae 48 (2018) via http://ami.ektf.hu/.
\end{abstract}

% 11T71  	Number theory - Algebraic coding theory; cryptography
% 05C20  	Combinatorics -Directed graphs (digraphs), tournaments
% 11B37  	Number theory - Recurrences 

\section{Introduction}

Public key cryptography began in $1976$ with a publication of Diffie and Hellman \cite{DH}, their fundamental work is {\it New direction in cryptography}. 
In the most cases, the security of a protocol is based on known hard questions in mathematics, and particularly in number theory. One of them is the discrete logarithm problem (in short, DLP). Let $p$ denote a large prime integer, say having more than a hundred of digits. If $a$ is a primitive root modulo $p$, and $b$ is a fixed integer not divisible by $p$, then it is difficult to compute the unknown $x$ such that 
\begin{equation}\label{christ}
a^x\equiv b\ \pmod{p}.
\end{equation} 
For example, the Diffie and Hellman method \cite{DH} and ElGamal signature \cite{ElGamal} are based on the supposition that this modular equation is intractable. It is easy to see that if 2 is also a primitive root modulo $p$, and $2^y\equiv a\ \pmod{p}$ can be efficiently solved, then (\ref{christ}) can be also efficiently solved. Hence it is sufficient to investigate the DLP with base 2. The present paper is also associated with this specification.

The first significant algorithm for solving the discrete logarithm problem was proposed by Shanks \cite{Shanks} in $1971$. 
Pohlig and Hellman \cite{PohligHellman} 
published an improved algorithm in $1978$. In the same year, other methods were suggested by Pollard \cite{Pollard}. But until now, no polynomial time algorithm is known. This fact justifies the efforts made by researchers to obtain advances in this mathematical field. 

In $2013$ two of the authors \cite{KhadirSzalay} studied a special recurrent integer sequence $(u_n)_{n\in\mathbb N}$ which can be used in solving the discrete logarithm problem when some favorable conditions are satisfied. More precisely, let $p$ and $q$ be odd primes such that $p=2q+1$, and 2 is a primitive root modulo $p$. Further let $u_0=b$, $1\le b\le q$, and 
\begin{equation}\label{seq1}
u_{n+1} = \begin{cases} u_n/2, &\mbox{if $u_n$ is even,}  \\
(p-u_n)/2, &\mbox{if $u_n$ is odd.} \end{cases}
\end{equation}
They proved that if $n_0$ is the smallest positive integer such that $u_{n_0}=1$, then $x_{n_0}$ is a solution of the discrete logarithm problem $2^x\equiv b \;(\bmod\, p)$. Here $x_0=0$, further
\begin{equation}\label{seq2}
x_{n+1} = \begin{cases} x_n+1 \;(\bmod\, p), &\mbox{if $u_n$ is even,}  \\
x_n+1+q \;(\bmod\, p), &\mbox{if $u_n$ is odd.} \end{cases}
\end{equation}
Consequently, the designers of cryptosystems must avoid the situation of small $n_0$.

The connection between DLP and the sequences (\ref{seq1}), (\ref{seq2}) motivated us to investigate the graph generated by (\ref{seq1}) if one considers it as a map on the set $S_p=\{1,2,\dots,p-1\}$. In this work, we principally concentrated on the structure of the aforementioned graph. Here we assume only the primality of $p$, and we do not suppose the primality of $q$ in $p=2q+1$. It turned out that our graphs are so-called sunlet graphs (see, for example \cite{Fu}), and we discovered and described many properties of them.

Our paper is organized as  follows. In Section 2 we define the map which induces the graph denoted by ${\cal G}_p$. Then we investigate the properties of the graph. Section 3 is devoted to provide some examples and remarks.

\section{The map and its properties}

Fix an odd prime $p$, and then the set $S_p=\{1,2,\dots,p-1\}$. Consider the map
\begin{equation}\label{elag1}
u(n+1) = \begin{cases} u(n)/2, &\mbox{if $u(n)$ is even,}  \\
(p-u(n))/2, &\mbox{if $u(n)$ is odd} \end{cases}
\end{equation}
on $S_p$. The map $u$ induces a digraph ${\cal G}_p$, such that there is an edge from $x$ to $y$ exactly when $u(x)=y$. In this paper, we describe the structure and some properties of the graph induced by (\ref{elag1}). As an illustration, the graph belonging to $p=17$ is drawn in Fig.~\ref{fig:graph_p17}.

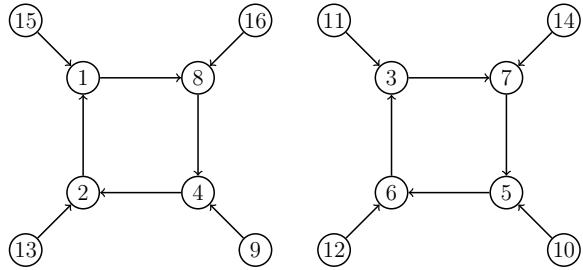
\begin{figure}[h!] \centering \scalebox{0.72}{
\begin{tikzpicture}[scale=0.75]
\node[vertex] (1) at (180-45:2) {1};
\node[vertex] (2) at (180-1*90-45:2) {8};
\node[vertex] (3) at (180-2*90-45:2) {4};
\node[vertex] (4) at (180-3*90-45:2) {2};

\foreach \from/\to in {1/2,2/3,3/4,4/1}
\path[edge] (\from) edge (\to);

\node[vertex] (a1) at (180-0*90-45:4) {15};
\node[vertex] (a2) at (180-1*90-45:4) {16};
\node[vertex] (a3) at (180-2*90-45:4) {9};
\node[vertex] (a4) at (180-3*90-45:4) {13};

\draw[edge] (a1) to (1);
\draw[edge] (a2) to (2);
\draw[edge] (a3) to (3);
\draw[edge] (a4) to (4);
\end{tikzpicture}\qquad
\begin{tikzpicture}[scale=0.75]
\node[vertex] (1) at (180-45:2) {3};
\node[vertex] (2) at (180-1*90-45:2) {7};
\node[vertex] (3) at (180-2*90-45:2) {5};
\node[vertex] (4) at (180-3*90-45:2) {6};

\foreach \from/\to in {1/2,2/3,3/4,4/1}
\path[edge] (\from) edge (\to);

\node[vertex] (a1) at (180-0*90-45:4) {11};
\node[vertex] (a2) at (180-1*90-45:4) {14};
\node[vertex] (a3) at (180-2*90-45:4) {10};
\node[vertex] (a4) at (180-3*90-45:4) {12};

\draw[edge] (a1) to (1);
\draw[edge] (a2) to (2);
\draw[edge] (a3) to (3);
\draw[edge] (a4) to (4);
\end{tikzpicture}\quad
}
\caption{Sunlet subgraphs in case $p=17$}
    \label{fig:graph_p17}
\end{figure}
Define
$$c_p=\frac{p-1}{2\,\ord_p(4)},$$
that is clearly integer. We prove the following theorem.
\begin{theorem}\label{main}
The graph ${\cal G}_p$ splits into $c_p$ connected isomorphic subgraphs. Each subgraph contains a cycle with length $L_p=\ord_p(4)$, and each vertex of the cycle possesses two incoming edges.
\end{theorem}

\noindent First we justify a lemma which has an important
corollary.

\begin{lemma} \label{l1}
Suppose that $u(x)=a$ and $u(y)=b$ hold for some $x,y,a,b\in S_p$. Then
\begin{equation}\label{eq1}
ay\equiv(-1)^{y-x}bx\pmod{p}.
\end{equation}
\end{lemma}

\begin{proof}
If $x$ and $y$ have the same parity, then either $a=x/2$ and $b=y/2$, or $a=(p-x)/2$ and $b=(p-y)/2$. Hence either
$ay=a\cdot2b=2a\cdot b=bx$, or $ay=a(p-2b)\equiv b(p-2a)=bx\pmod{p}$, respectively.

Assume now that $x\not\equiv y\pmod{2}$. It leads either $a=x/2$ and $b=(p-y)/2$, or $a=(p-x)/2$ and $b=y/2$.
In the first case we see $ay=a(p-2b)=ap-bx\equiv -bx\pmod{p}$, while in the second case we have
$ay=a\cdot2b\equiv-b(p-2a)=-bx\pmod{p}$.

Then the statement is clearly comes from the previous arguments.
\end{proof}

Now we give a direct consequence of Lemma \ref{l1}.

\begin{corollary}\label{c1}
Under the same conditions
\begin{equation}\label{eq2}
a\equiv(-1)^{y-x}bxy^{-1}\pmod{p}
\end{equation}
holds.
\end{corollary}

\noindent Now we give the proof of Theorem \ref{main}, which is split into a few parts called observations. Put $q=(p-1)/2$. Note that the map $u$ does not possess fixed points.

\begin{obs}\label{o0}
If $u(x)=u(y)$ holds for some $x\neq y$, then $x+y=p$.
\end{obs}

\begin{proof}
Since $x\ne y$, we see that the parity of $x$ differs the parity of $y$. Thus either
$$
\frac{x}{2}=\frac{p-y}{2}\qquad {\rm or}\qquad \frac{p-x}{2}=\frac{y}{2}
$$
follows, both options admit $x+y=p$.
\end{proof}

\begin{obs}\label{o1}
The equation $u(x)=a$ is soluble if and only if $a\le q$, and in this case there exist exactly two solutions.
\end{obs}

\begin{proof}
Assume that $x$ and $a$ satisfy $u(x)=a$.
If $x\in S_p$ is even, then $u(x)=x/2\le(p-1)/2=q$. Contrary, if $x$ is odd, then $u(x)=(p-x)/2\le(p-1)/2=q$. On the other hand, $u(2a)=a$ and $u(p-2a)=a$ hold. By Observation \ref{o0} no third solution to the equation.
\end{proof}

Note that exactly one of $2a$ and $p-2a$ is larger than $q$. Let $S_p^\ell=\{1,2,\dots,q\}$ and $S_p^u=\{q+1,q+2,\dots,p-1\}$. Clearly $S_p^\ell\cup S_p^u=S_p$, and $|S_p^\ell|=|S_p^\ell|$.
Hence, using graph theoretical terminology, we obtain the following information about the structure of ${\cal G}_p$:
{\it the elements of $S_p^\ell$ form cycle(s), further each element of $S_p^u$ goes to an appropriate element of $S_p^\ell$ such that different elements of $S_p^u$ go different elements of $S_p^\ell$. In other words, ${\cal G}_p$ consists of sunlet graph(s) (or sun graph(s)).}

In the next step we show that the sunlet graphs included in ${\cal G}_p$ are isomorphic.

\begin{obs}\label{o2}
If ${\cal G}_p$ consists of at least two connected sunlet graphs, then all the sunlet graphs are isomorphic.
\end{obs}

\begin{proof}
Obviously it is sufficient to prove that two cycles have the same length.
Take two cycles, saying $x_1,x_2,\dots,x_n$ and $y_1,y_2,\dots,y_k$, where $n\ge2$ and $k\ge2$. Without loss of generality we may assume that $k\le n$. By Corollary \ref{c1} the following congruences hold modulo $p$.
\begin{eqnarray*}
y_2&\equiv&(-1)^{x_1-y_1}y_1x_2x_1^{-1}, \\
y_3&\equiv&(-1)^{x_2-y_2}y_2x_3x_2^{-1}, \\
&\vdots& \\
y_k&\equiv&(-1)^{x_{k-1}-y_{k-1}}y_{k-1}x_kx_{k-1}^{-1}, \\
y_{k+1}=y_1&\equiv&(-1)^{x_{k}-y_{k}}y_{k}x_{k+1}x_{k}^{-1}. 
\end{eqnarray*}
The product of all the congruences above returns with
$$
1\equiv (-1)^{x_\sigma-y_\sigma}x_{k+1}x_1^{-1}\pmod{p},
$$
where $x_\sigma=\sum_{i=1}^k x_i$ and $y_\sigma=\sum_{i=1}^ky_i$. Thus
$$
x_1\equiv (-1)^{x_\sigma-y_\sigma}x_{k+1}\pmod{p}.
$$
In accordance with the parity of exponent ${x_\sigma-y_\sigma}$, we have either
$x_{k+1}=x_1$ or $x_{k+1}=p-x_1$. But the second case cannot be occurred because it leads to a contradiction by $q\ge x_{k+1}=p-x_1>q$. Subsequently, $x_{k+1}=x_1$, and then $n=k$.
\end{proof}

A direct consequence is the following statement.
\begin{corollary}
$L_p\mid p-1$.
\end{corollary}

\begin{obs}\label{o3}
$L_p=\ord_p(4)$.
\end{obs}

\begin{proof}
The formula (\ref{elag1}) of map $u$ implies
\begin{equation}\label{cong}
u(x)\equiv\pm\frac{x}{2}\pmod{p},
\end{equation}
where the minus sign is occurring  exactly if $x$ is odd. Applying (\ref{cong}) consecutively for the cycle $x_1,x_2,\dots,x_{L_p}$
it leads to
$$
x_1\equiv(-1)^t\frac{x_1}{2^{L_p}}\pmod{p},
$$
where $t$ is a suitable non-negative integer, showing the number of odd entries of map $u$. Equivalently we have
$$
2^{L_p}\equiv(-1)^t\pmod{p},
$$
and then
$$
4^{L_p}\equiv 1\pmod{p}.
$$
Thus $\ord_p(4)\mid L_p$. To show the reverse relation $L_p\mid\ord_p(4)$ we assume $\ord_p(4)> L_p$. Let $s\ge1$ and $0\le r< L_p$ two non-negative integers such that $\ord_p(4)=sL_p+r$, where $r\ne0$ holds if $s=1$.
Consider now the sequence
$$
x_1,x_2,\dots, x_{L_p};x_1,x_2,\dots, x_{L_p};\dots;x_1,x_2,\dots, x_{L_p};x_1,x_2,\dots, x_{r},
$$
assuming that here the cycle $x_1,x_2,\dots, x_{L_p}$ occurs $s$ times.
For a suitable $\tau$ we see
$$
x_r\equiv(-1)^\tau\frac{x_1}{2^{\ord_p(4)}}\pmod{p},
$$
and then squaring both sides it follows that
$$
x_r^2\equiv x_1^2\pmod{p}.
$$
It provides either $x_r+x_1=p$ which contradicts the facts that neither $x_1$ nor $x_r$ exceeds $q$, or $x_r=x_1$ which leads to $qL_p=\ord_p(4)$, that is $L_p\mid \ord_p(4)$. Together with $\ord_p(4)\mid L_p$ we conclude $L_p=\ord_p(4)$, and the proof is complete.
\end{proof}

\section{Examples and remarks}

1. Let $p=31$. Now $L_{31}=\ord_{31}(4)=5$ is the length of the cycles. The number of connected subgraphs is $c_{31}=30/(2\cdot 5)=3$. The corresponding graph is drawn here.

\begin{figure}[h!] \centering \scalebox{0.72}{
\begin{tikzpicture}[scale=0.63]
\node[vertex] (1) at (18+3*72-1*360/5:2) {1};
\node[vertex] (2) at (18+3*72-2*360/5:2) {15};
\node[vertex] (3) at (18+3*72-3*360/5:2) {8};
\node[vertex] (4) at (18+3*72-4*360/5:2) {4};
\node[vertex] (5) at (18+3*72-5*360/5:2) {2};

\foreach \from/\to in {1/2,2/3,3/4,4/5, 5/1}
\path[edge] (\from) edge (\to);

\node[vertex] (a1) at (18+3*72-1*360/5:4) {29};
\node[vertex] (a2) at (18+3*72-2*360/5:4) {30};
\node[vertex] (a3) at (18+3*72-3*360/5:4) {16};
\node[vertex] (a4) at (18+3*72-4*360/5:4) {23};
\node[vertex] (a5) at (18+3*72-5*360/5:4) {27};

\draw[edge] (a1) to (1);
\draw[edge] (a2) to (2);
\draw[edge] (a3) to (3);
\draw[edge] (a4) to (4);
\draw[edge] (a5) to (5);
\end{tikzpicture}\quad
\begin{tikzpicture}[scale=0.63]
\node[vertex] (1) at (18+3*72-1*360/5:2) {3};
\node[vertex] (2) at (18+3*72-2*360/5:2) {14};
\node[vertex] (3) at (18+3*72-3*360/5:2) {7};
\node[vertex] (4) at (18+3*72-4*360/5:2) {12};
\node[vertex] (5) at (18+3*72-5*360/5:2) {6};

\foreach \from/\to in {1/2,2/3,3/4,4/5, 5/1}
\path[edge] (\from) edge (\to);

\node[vertex] (a1) at (18+3*72-1*360/5:4) {25};
\node[vertex] (a2) at (18+3*72-2*360/5:4) {28};
\node[vertex] (a3) at (18+3*72-3*360/5:4) {17};
\node[vertex] (a4) at (18+3*72-4*360/5:4) {24};
\node[vertex] (a5) at (18+3*72-5*360/5:4) {19};

\draw[edge] (a1) to (1);
\draw[edge] (a2) to (2);
\draw[edge] (a3) to (3);
\draw[edge] (a4) to (4);
\draw[edge] (a5) to (5);
\end{tikzpicture}\quad
\begin{tikzpicture}[scale=0.63]
\node[vertex] (1) at (18+3*72-1*360/5:2) {5};
\node[vertex] (2) at (18+3*72-2*360/5:2) {13};
\node[vertex] (3) at (18+3*72-3*360/5:2) {9};
\node[vertex] (4) at (18+3*72-4*360/5:2) {11};
\node[vertex] (5) at (18+3*72-5*360/5:2) {10};

\foreach \from/\to in {1/2,2/3,3/4,4/5, 5/1}
\path[edge] (\from) edge (\to);

\node[vertex] (a1) at (18+3*72-1*360/5:4) {21};
\node[vertex] (a2) at (18+3*72-2*360/5:4) {26};
\node[vertex] (a3) at (18+3*72-3*360/5:4) {18};
\node[vertex] (a4) at (18+3*72-4*360/5:4) {22};
\node[vertex] (a5) at (18+3*72-5*360/5:4) {20};

\draw[edge] (a1) to (1);
\draw[edge] (a2) to (2);
\draw[edge] (a3) to (3);
\draw[edge] (a4) to (4);
\draw[edge] (a5) to (5);
\end{tikzpicture}\quad
}
    \caption{Sunlet subgraphs in case of $p=31$}
    \label{fig:graph_p31}
\end{figure}
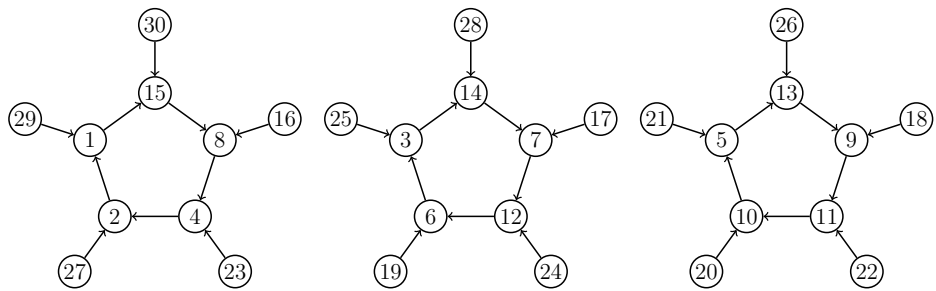

\noindent 2. Let $p=5419$. Now $L_{5419}=\ord_{5419}(4)=21$ is relatively a very small value for the length of the cycles, and primes having such a property are unavailable for cryptographic purposes. The number of connected subgraphs is $c_{5419}=129$.

\bigskip

{\bf Acknowledgments}
This paper was written when the first author visited the Institute of Mathematics, University of Sopron, and the Department of Mathematics and Informatics, J.~Selye University. He expresses his gratitude both departments for their hospitality.

\end{document}